\documentclass[12pt]{article}

\usepackage{amssymb}
\usepackage{amsfonts}
\usepackage{amsmath}
\usepackage{amsthm}
\usepackage[margin=1.00in]{geometry}
\usepackage{enumerate}
\usepackage{amstext}
\usepackage{layout}

\numberwithin{equation}{section}
\newtheorem{theorem}{Theorem}

\newtheorem{lemma}{Lemma}

\newtheorem{remark}{Remark}
\newtheorem{problem}{Problem}

\begin{document}

\begin{center}
{\Large \bf Some Results on the Central Limit Theorem for Subsequences in Banach Spaces}

\vskip 0.5cm

{\bf Deli Li$^{1}$\footnote{Corresponding author: Deli Li
	(email: dli@lakeheadu.ca)}} and Han-Ying Liang$^{2}$

\vskip 0.2cm

$^1$Department of Mathematical Sciences, Lakehead
University, Thunder Bay, Ontario, Canada\\
$^2$Department of Mathematics, Tongji University, Shanghai, China
\end{center}

\noindent {\bf Abstract}. Let $\{X, X_{n}; n \geq 1 \}$ be a 
sequence of i.i.d. $\mathbf{B}$-valued random variables and 
set $S_{n} = \sum_{i=1}^{n}X_{i},~n \geq 1$. 
This note is devoted to study the classical central limitr theorem for 
subsequences of sums of i.i.d. $\mathbf{B}$-valued random variables. We show that, under the assumption that $\mathbf{B}$ 
is of cotype $2$ space, $\left(\frac{S_{n}}{\sqrt{n}} \right)_{n \geq 1}$ converges weakly if and only if
$\left(\frac{S_{m_{n}}}{\sqrt{{m_{n}}}} \right)_{n \geq 1}$ converges weakly for a subsequence $\{m_{n}; ~n \geq 1\}$ of positive integers.
We conjecture that this result is false if $\mathbf{B}$ is not of cotype $2$ space. In addition, we show that, if $\left(\frac{S_{m_{n}}}{\sqrt{{m_{n}}}} \right)_{n \geq 1}$ converges weakly for a subsequence $\{m_{n}; ~n \geq 1\}$ of positive integers. and $\left(\frac{S_{n}}{\sqrt{n}} \right)_{n \geq 1}$ does not converge weakly, then $\displaystyle \left(S_{n}/a_{n} \right)_{n \geq 1}$ does not converge weakly to a non-degenerate probability measure for 
any sequence $\{a_{n}; n \geq 1 \}$ of positive real numbers.

\medskip

\noindent {\bf Keywords}~~Central limit theorem $\cdot$ Cotype $2$ space $\cdot$
Hilbert space $\cdot$ Subseqences $\cdot$ Sums of i.i.d.
B-valued random variables 

\medskip

\noindent {\bf MSC}: 60F05; 60B12; 60G50

\section{Introduction}

This note is devoted to study the classical central limitr theorem for 
subsequences of sums of independent and identically distributed (i.i.d.) 
$\mathbf{B}$-valued random variables. 
To better explain the motivation behind this work, we will first introduce 
some concepts and definitions. Throughout this note, let 
$(\mathbf{B}, \| \cdot \| )$ be a real separable Banach space with topological 
dual space $\mathbf{B}^{*}$,  $(\Omega, \mathcal{F}, \mathbb{P})$ a rich 
complete probability space, and $\mathcal{B}$ the Borel $\sigma$-algebra 
generated by the class of open subsets of $\mathbf{B}$ determined by $\|\cdot\|$. 
A {\bf B}-valued random variable $X$ is defined as a measurable function from 
$(\Omega, \mathcal{F})$ into $(\mathbf{B}, \mathcal{B})$. The symbol 
$N(0, \sigma^{2})$ stands for a real-valued normal distribution with mean 
of $0$ and variance of $\sigma^{2}$. Let $\mathbb{N}$ be the set of all 
positive integers. We say that $\{m_{n}; ~n \in \mathbb{N}\}$ is a subsequence 
of $\mathbb{N}$ if $m_{n} \in \mathbb{N}$ and $m_{n} < m_{n+1}$ for all 
$n \in \mathbb{N}$. Let $\{X, X_{n};~n \in \mathbb{N} \}$ be a sequence 
of i.i.d. $\mathbf{B}$-valued random variables and, as usual, 
set $S_{n} = \sum_{i=1}^{n}X_{i},~n \in \mathbb{N}$. 
We say that $X$ satisfies the classical central limit 
theorem (and write $X \in \mbox{CLT}$) if
\[
\left(\frac{S_{n}}{\sqrt{n}} \right)_{n \in \mathbb{N}} ~\mbox{converges weakly}.
\]
We say that $X$ satisfies the central limit theorem 
for subsequence (and write $X \in \mbox{CLTSS}$) if 
\[
\left(\frac{S_{m_{n}}}{\sqrt{{m_{n}}}} \right)_{n \in \mathbb{N}} ~
\mbox{converges weakly for a subsequence}~\{m_{n}; ~n \in \mathbb{N} \}
~\mbox{of}~\mathbb{N}.
\]

The classical CLT of the sums of i.i.d. $\mathbf{B}$-valued random variables 
has been deeply studied by many mathematicians, and many profound results have 
been obtained. Regarding this topic, Ledoux and Talagrand's book [10] provides 
a comprehensive treatment of the classical CLT in the context of Banach spaces 
and lists almost all authors who have made significant contributions to the 
classical CLT in a Banach space setting.

As is well known, in mathematical analysis, if the entire sequence converges, 
then every subsequence must converge; conversely, the convergence of a certain 
subsequence does not necessarily mean that the entire sequence converges.
There are some studies in the literature on the limit theorems for subsequences 
of partial sums of i.i.d. (real-valued or $\mathbf{B}$-valued) random variables. 
For example, Gut [4] studied complete convergence in the law of large numbers 
for subsequences, Chow {\it et al.} [2], Fukuyama and Takeuchi [3], 
Gut [5], Qualls [12], Schwabe and Gut [13], Weber [15],  and among others studied the 
law of the iterated logarithm for subsequences, Li {\it et al.} [11] and Song [14]
studied the CLT for subsequences, etc. 

Let's look at the situation corresponding to the law of large numbers. 
Let $0 < p < 2$. We say that $X$ satisfies the strong law of large numbers
(resp., the weak law of large numbers) for the normalize sequence $\{n^{1/p};~ n \in \mathbb{N} \}$
(and write $X \in \mbox{SLLN}_{p}$ (resp., $X \in \mbox{WLLN}_{p}$))
if 
\[
\frac{S_{n}}{n^{1/p}} \stackrel{\mbox{a.s.}}{\longrightarrow} 0 
~\left(\mbox{resp.}, \frac{S_{n}}{n^{1/p}} \stackrel{\mathbb{P}}{\longrightarrow} 0 \right).
\]
Here and below, ``$\stackrel{\mbox{a.s.}}{\longrightarrow}$" and 
``$\stackrel{\mathbb{P}}{\longrightarrow}$" stand for convergence 
in almost surely and convergence in probability, respectively.
 We say that $X$ satisfies the strong law of large numbers
(resp., the weak law of large numbers) for a subsequence of the normalize sequence $\{n^{1/p};~ n \in \mathbb{N} \}$
(and write $X \in \mbox{SLLNSS}_{p}$ (resp., $X \in \mbox{WLLNSS}_{p}$))
if 
\[
\frac{S_{m_{n}}}{m_{n}^{1/p}} \stackrel{\mbox{a.s.}}{\longrightarrow} 0 
~\left(\mbox{resp.}, \frac{S_{m_{n}}}{m_{n}^{1/p}} \stackrel{\mathbb{P}}{\longrightarrow} 0 \right) 
~\mbox{for a subsequence}~\{m_{n}; ~n \in \mathbb{N} \}
~\mbox{of}~\mathbb{N}.
\]
For each given $0 < p < 2$,  it is well-known that, for any Banach space $\mathbf{B}$,
\[
X \in \mbox{SLLN}_{p} \Rightarrow X \in \mbox{SLLNSS}_{p}
~~\mbox{and}~~X \in \mbox{SLLN}_{p} \nLeftarrow X \in \mbox{SLLNSS}_{p}
\]
and
\[
X \in \mbox{WLLN}_{p} \Rightarrow X \in \mbox{WLLNSS}_{p}
~~\mbox{and}~~X \in \mbox{WLLN}_{p} \nLeftarrow X \in \mbox{WLLNSS}_{p}.
\]
However, for the CLT, the situation is somewhat unexpected. For $\mathbf{B} = (-\infty, \infty)$, 
it follows from Proposition 1 of Song [14] that 
\begin{equation}
X \in \mbox{CLTSS} ~~\mbox{if and only if}~~\mathbb{E}X = 0~\mbox{and}~ \mathbb{E}X^{2} < \infty \tag{1}
\end{equation}
and hence,
\begin{equation}
X \in \mbox{CLT} ~~\mbox{if and only if}~~X \in \mbox{CLTSS}. \tag{2}
\end{equation}
Motivated by Proposition 1 of Song [14], Li {\it el al.} [11] studied the CLT 
for subsequences when $\{\sqrt{n}; n \in \mathbb{N} \}$ is replaced by
$\{\sqrt{na_{n}};~n \in \mathbb{N} \}$ with $\lim_{n \rightarrow \infty}
a_{n} = \infty$. Theorem 2.1 of Li {\it et al.} [11] asserts that, for given positive nondecreasing
sequence $\{a_{n}; n \in \mathbb{N} \}$ with $\lim_{n \rightarrow \infty}
a_{n} = \infty$ and $\lim_{n \rightarrow \infty} a_{n+1}/a_{n} = 1$
and given nondecreasing function $h(\cdot):~(0, \infty) \rightarrow
(0, \infty)$ with $\lim_{x \rightarrow \infty} h(x) = \infty$, there
exists a sequence $\{X, X_{n}; n \in \mathbb{N} \}$ of symmetric real-valued i.i.d. random
variables such that $\mathbb{E}h(|X|) = \infty$ and
$\displaystyle S_{n_{k}}/\sqrt{n_{k}a_{n_{k}}} \Rightarrow N(0, 1)$
for some subsequence
$\{n_{k}; k \in \mathbb{N} \}$ of $\mathbb{N}$.
In particular, for given $0 < p < 2$ and given nondecreasing function
$h(\cdot): (0, \infty) \rightarrow (0, \infty)$ with $\lim_{x
	\rightarrow \infty} h(x) = \infty$, there exists a sequence $\{X,
X_{n}; n \in \mathbb{N} \}$ of symmetric i.i.d. random variables such that
$\mathbb{E}h(|X|) = \infty$ and $\displaystyle S_{n_{k}}/n_{k}^{1/p}
\Rightarrow N(0, 1)$ for some subsequence $\{n_{k}; k \in \mathbb{N} \}$ of $\mathbb{N}$.

From Proposition 1 of Song [14], it is easy to see that, (1) and (2) 
hold for any finite-dimensional Banach space $\mathbf{B}$. Let $\mathbf{c}_{0}$ 
be the Banach space of all sequences of real numbers converging to zero equipped 
with the uniform norm. Interestingly, Song [14] constructed a sequence of 
$\mathbf{c}_{0}$-valued i.i.d. random variable $\{X, X_{n}; n \in \mathbb{N} \}$ 
such that 
\[
X \in \mbox{CLTSS}, ~\mbox{but}~ X \notin \mbox{CLT}.
\]

Motivated by the facts mentioned above, we now address the following two problems. 

\begin{problem}
For any infinite-dimensional separable Banach space $\mathbf{B}$, are (1) and (2) not true? 
\end{problem}

\begin{problem}
As we know, $\mathbb{E}X = 0$ and $\mathbb{E}\|X\|^{2} < \infty$ are neither sufficient 
nor necessary for $X \in \mbox{CLT}$ to hold (see, e.g., Ledoux and Talagrand [10, Chapter 10]). 
Since $X \in \mbox{CLT} \Rightarrow X \in \mbox{CLTSS}$ and, 
from the counterexample constructed by Song [14], $X \in \mbox{CLT} \nLeftarrow X \in \mbox{CLTSS}$, 
are $\mathbb{E}X = 0$ and $\mathbb{E}\|X\|^{2} < \infty$ also neither sufficient 
nor necessary for $X \in \mbox{CLTSS}$ to hold?
\end{problem}

Let $\{a_{n}; ~ n \in \mathbb{N} \}$ be a nondecreasing sequence of positive real numbers such that
\[
\liminf_{n \rightarrow \infty} \frac{a_{n}}{\sqrt{n}} = 1~~\mbox{and}
~~\limsup_{n \rightarrow \infty}\frac{a_{n}}{\sqrt{n}} = \infty.
\]
Let $\{\Phi_{n}; n \in \mathbb{N}\}$ be a sequence of real-valued random variables such that,
for each $n \in \mathbb{N}$, the probability distribution of $\Phi_{n}$ is given by $N\left(0, a^{2}_{n} \right)$. Then, one can easily see that
\[
\left(\frac{\Phi_{k_{n}}}{\sqrt{k_{n}}}\right)_{n \in \mathbb{N}} \mbox{converges weakly to}~N(0, 1)~\mbox{for some subsequence}~\{k_{n}; ~n \in \mathbb{N} \}
~\mbox{of}~\mathbb{N}
\]
and 
\[
\left(\frac{\Phi_{n}}{\sqrt{n}}\right)_{n \in \mathbb{N}} \mbox{does not converge weakly to any real-valued random variable.}
\]
Clearly, if $\left \{\sqrt{n}; ~n \in \mathbb{N} \right \}$ is replaced by 
$\left \{a_{n}; ~n \in \mathbb{N} \right \}$, then we have
\[
\left(\frac{\Phi_{n}}{a_{n}}\right)_{n \in \mathbb{N}} \mbox{converges weakly to}~ N(0, 1).
\]
Thus, motivated by this example, our next problem can be addressed as follows. 

\begin{problem}
For the counterexample constructed by Song [14], although
$\left(\frac{S_{n}}{\sqrt{n}} \right)_{n \in \mathbb{N}}$ does not 
converge weakly, is it possible that 
$\left(\frac{S_{n}}{a_{n}} \right)_{n \in \mathbb{N}}$
does converge weakly to a nondegenerate measure for some normalize sequence
$\{a_{n}; n \in \mathbb{N} \}$ of positive numbers?
\end{problem}

The plan of this note is as follows. Our main results, Theorems 1 and 2
are stated in Section 2. From Theorems 1 and 2, we see that the answers to Problems 1 and 3 
above are negative. Combining our Theorem 1 (ii) and a theorem established by Jain [8], 
the answer to Problem 2 above is positive. The proofs of Theorems 1 and 2 are given in Sections 3 and 4 respectively. The main findings of this note and their implications are summarized in Section 5.

\section{Statement of the Main Results}

In this section, we present the main results of this note. For the concept of Banach space types (especially cotype $2$ and type $2$), see the book of Ledoux and Talagrand [10].

\begin{theorem}
Let $X$ be a $\mathbf{B}$-valued random variable. 
Then we have the following two conclusions.

\noindent
{\bf (i)}. If $\mathbf{B}$ is of both cotype $2$ and type $2$ space, then 
\begin{equation}
X \in \mbox{CLT} \Leftrightarrow X \in \mbox{CLTSS} 
\Leftrightarrow \mathbb{E}X = 0 ~\mbox{and}~ \mathbb{E} \|X\|^{2} < \infty.
\tag{3}
\end{equation}

\noindent
{\bf (ii)}. If $\mathbf{B}$ is of both cotype $2$ and non-type $2$ space, then 
\begin{equation}
X \in \mbox{CLT} \Leftrightarrow X \in \mbox{CLTSS} 
\stackrel{\Rightarrow}{\nLeftarrow}
\mathbb{E}X = 0 ~\mbox{and}~ \mathbb{E} \|X\|^{2} < \infty.
\tag{4}
\end{equation}
\end{theorem}

\begin{remark}
 Let $\mathbf{B}$ be of cotype $2$ space. Then, from Theorem 1, we see that
 $X \in \mbox{CLT} \Leftrightarrow X \in \mbox{CLTSS}$.  As we know, $l^{p}$ spaces with 
 $1 \leq p \leq 2$ are cotype $2$.  More generally, $L^{p}$ spaces (which include $l^{p}$) 
 with $1 \leq p \leq 2$ are also cotype $2$. Thus, from our Theorem 1, our answer to 
 Problem 1 is negative. 
\end{remark}

\begin{remark}
	Kwapi\'{e}n [9, Proposition 3.1] showed that a Banach space is isomorphic to 
	a Hilbert space if and only if it is of both cotype $2$ and type $2$ space.
	Thus, combining this result and our Theorem 1 (i), we conclude that if Banach space $\mathbf{B}$ is isomorphic to a Hilbert space, then (3) holds.
\end{remark}

\begin{remark}
From the conclusion of Theorem 1 (ii), we see that $\mathbb{E}X = 0$ and 
 $\mathbb{E}\|X\|^{2} < \infty$ are not sufficient for $X \in \mbox{CLTSS}$ 
 to hold if $\mathbf{B}$ is of cotype $2$ and not of type $2$ space. 
 Since $X \in \mbox{CLT} \Rightarrow X \in \mbox{CLTSS}$, and $\mathbb{E}X = 0$ and $\mathbb{E}\|X\|^{2} < \infty$ are not necessary for $X \in \mbox{CLT}$ to hold, $\mathbb{E}X = 0$ and $\mathbb{E}\|X\|^{2} < \infty$ are also not necessary conditions for $X \in \mbox{CLTSS}$ to hold. 
Thus, our answer to Problem 2 is positive; that is, $\mathbb{E}X = 0$ and $\mathbb{E}\|X\|^{2} < \infty$ are neither sufficient 
 nor necessary for $X \in \mbox{CLTSS}$ to hold.
\end{remark}

\begin{theorem}
	Let $X$ be a $\mathbf{B}$-valued random variable such that $X \in \mbox{CLTSS}$ and 
	$X \notin \mbox{CLT}$. Then $\displaystyle \left(S_{n}/a_{n} \right)_{n \in \mathbb{N}}$ 
	does not converge weakly to a non-degenerate probability measure for 
	any sequence $\{a_{n}; n \in \mathbb{N} \}$ of positive real numbers.
\end{theorem}

\begin{remark}
Clearly, Theorem 2 provides a negative answer to Problem 3 in Section 1.
For the counterexample constructed by Song [14],  $X \in \mbox{CLTSS}$ 
and $X \notin \mbox{CLT}$, by Theorem 2, we conclude that $\displaystyle \left(S_{n}/a_{n} \right)_{n \in \mathbb{N}}$ 
does not converge weakly to a non-degenerate probability measure for 
any sequence $\{a_{n}; n \in \mathbb{N} \}$ of positive real numbers.
\end{remark}

\section{Proof of Theorem 1}

To prove Theorem 1, we first introduce two definitions. 
A $\mathbf{B}$-valued random variable $G$ is said to be Gaussian 
if the real-valued random variable $f(G)$ has a normal distribution for all
$f \in \mathbf{B}^{*}$ (see, e.g., Ledoux and Talagrand [10, page 55]).
A $\mathbf{B}$-valued random variable $X$ is said to be 
pre-Gaussian if $\mathbb{E}f(X) = 0$ and $\mathbb{E}f^{2}(X) < \infty$
for all $f \in \mathbf{B}^{*}$ and there exists a $\mathbf{B}$-valued Gaussian 
random variable $G$ such that $\mathbb{E}f^{2}(X) = \mathbb{E}f^{2}(G)$ for 
all $f \in \mathbf{B}^{*}$ (see, e.g., Ledoux and Talagrand [10, page 260]).

For establishing Theorem 1, we need the following auxiliary result.

\begin{lemma}
	Let $X$ be a $\mathbf{B}$-valued random variable such that $X \in \mbox{CLTSS}$. 
	Then $X$ is pre-Gaussian.
\end{lemma}

\vskip 0.3cm

\noindent 
{\bf Proof.}~ Since $X \in \mbox{CLTSS}$, there exists a subsequence 
$\{m_{n}; ~n \in \mathbb{N} \}$ of $\mathbb{N}$ and a $\mathbf{B}$-valued 
random variable $U$ such that 
\begin{equation}
\left(\frac{S_{m_{n}}}{\sqrt{m_{n}}} \right)_{n \in \mathbb{N}} ~
\mbox{converges weakly to} ~U  
\tag{5}
\end{equation}
and hence, 
\begin{equation}
\left(f\left(\frac{S_{m_{n}}}{\sqrt{m_{n}}}\right) \right)_{n \in \mathbb{N}} ~
\mbox{converges weakly to} ~f(U)~\mbox{for all}~ f \in \mathbf{B}^{*}.
\tag{6}
\end{equation}
Note that, for every given $f \in \mathbf{B}^{*}$,
\[
\frac{\sum_{i=1}^{m_{n}}f\left(X_{i}\right)}{\sqrt{m_{n}}} =f\left(\frac{S_{m_{n}}}{\sqrt{m_{n}}}\right)
\]
and $\{f(X), f(X_{n}); n \in \mathbb{N} \}$ is a sequence of i.i.d. real-valued random variables. 
Thus, it follows from (6) that
\begin{equation}
\left(\frac{\sum_{i=1}^{m_{n}}f\left(X_{i}\right)}{\sqrt{m_{n}}} \right)_{n \in \mathbb{N}} ~
\mbox{converges weakly to} ~f(U)~\mbox{for all}~ f \in \mathbf{B}^{*};
\tag{7}
\end{equation}
i.e., $f(X) \in \mbox{CLTSS}$ for all $f \in \mathbf{B}^{*}$. Thus, by Proposition 1 of
Song [14], we have
\[
\mathbb{E}f(X) = 0 ~ \mbox{and}~ \mathbb{E}f^{2}(X) < \infty
~\mbox{for all}~ f \in \mathbf{B}^{*}
\]
and hence,
\begin{equation}
\left(\frac{\sum_{i=1}^{m_{n}}f\left(X_{i}\right)}{\sqrt{m_{n}}} \right)_{n \in \mathbb{N}} ~
\mbox{converges weakly to}~ N \left(0, \mathbb{E}f^{2}(X)\right)~\mbox{for all}~ f \in \mathbf{B}^{*}. \tag{8}
\end{equation}
Now (7) and (8) together ensure that the real-valued random variable $f(U)$ has 
a normal distribution with mean of $0$ and variance of $\mathbb{E}f^{2}(X)$ for all
$f \in \mathbf{B}^{*}$. Hence, we see that $U$ is a Gaussian $\mathbf{B}$-valued random variable
satisfying 
\[
\mathbb{E}U = 0 ~\mbox{and}~\mathbb{E}f^{2}(X) = \mathbb{E}f^{2}(U) ~\mbox{for all}~ f \in \mathbf{B}^{*}
\]
which ensures that $X$ is a pre-Gaussian $\mathbf{B}$-valued random variable. 
$\Box$

\vskip 0.2cm

\noindent {\bf Proof of Theorem 1.}~We first show that 
\begin{equation}
X \in \mbox{CLT} \Leftarrow X \in \mbox{CLTSS}  \tag{9}
\end{equation} 
if $\mathbf{B}$ is of cotype $2$ space. Since $X \in \mbox{CLTSS}$, it follows from Lemma 1 that $X$ is a 
pre-Gaussian $\mathbf{B}$-valued random variable. By a theorem of Jain [8] (see, e.g., Ledoux and Talagrand [10, Theorem 10.7]), $\mathbf{B}$ is of cotype $2$ space if and only if 
$X \in \mbox{CLT}$ whenever $X$ is pre-Gaussian. So we conclude immediately that, under the assumption that 
$\mathbf{B}$ is of cotype $2$ space, (9) holds and hence, $X \in \mbox{CLT} \Leftrightarrow X \in \mbox{CLTSS}$. 
Note that Jain [7] showed that
\[
X \in \mbox{CLT} \Rightarrow \mathbb{E}X = 0 ~\mbox{and}~ \mathbb{E} \|X\|^{2} < \infty
\]
if $\mathbf{B}$ is of cotype $2$ space. Thus, for cotype $2$ space $\mathbf{B}$, we have the property
\begin{equation}
X \in \mbox{CLT} \Leftrightarrow X \in \mbox{CLTSS} 
\Rightarrow \mathbb{E}X = 0 ~\mbox{and}~ \mathbb{E} \|X\|^{2} < \infty. 
\tag{10}
\end{equation}
Now note that Hoffmann-J{\o}rgensen and Pisier [6] showed that, for $\mathbf{B}$
to have the property
\begin{equation}
X \in \mbox{CLT~~whenever} ~\mathbb{E}X = 0 ~\mbox{and}~ \mathbb{E} \|X\|^{2} < \infty  \tag{11}
\end{equation}
it is necessary and sufficient that $\mathbf{B}$ must be of type $2$ space. Thus, 
we deduce the following two conclusions:

\noindent
{\bf (i).} If $\mathbf{B}$ is of both cotype $2$ and type $2$ space, then (3) follows from (10) and (11).
 
\noindent
{\bf (ii).} If $\mathbf{B}$ is of both cotype $2$ and non-type $2$ space, then it
follows from (10) and (11) that
\[
\mathbb{E}X = 0 ~\mbox{and}~ \mathbb{E} \|X\|^{2} < \infty
\nRightarrow X \in \mbox{CLT} \Leftrightarrow X \in \mbox{CLTSS} 
\Rightarrow \mathbb{E}X = 0 ~\mbox{and}~ \mathbb{E} \|X\|^{2} < \infty;
\]
i.e., (4) holds. The proof of Theorem 1 is complete. ~$\Box$

\section{Proof of Theorem 2}

The following preliminary lemma is used to prove Theorem 2.

\begin{lemma}
Let $\{Y_{n}; n \in \mathbb{N} \}$ be a sequence of $\mathbb{B}$-valued random variables
such that 
\begin{equation}
\left(Y_{n} \right)_{n \in \mathbb{N}} ~\mbox{converges weakly to some $\mathbf{B}$-valued random variable} ~W_{1}~
\mbox{with}~ \mathbb{P}\left(W_{1} = 0\right) < 1.
\tag{12}
\end{equation} 
Let $\{\alpha_{n}; n \in \mathbb{N} \}$ be a sequence of positive real numbers such that
\begin{equation}
\left(\alpha_{n}Y_{n} \right)_{n \in \mathbb{N}} ~\mbox{converges weakly to some $\mathbf{B}$-valued random variable} ~W_{2}~
\mbox{with}~ \mathbb{P}\left(W_{2} = 0\right) < 1.
\tag{13}
\end{equation} 
Then there exists $0 < c < \infty$ such that
\begin{equation}
\lim_{n \rightarrow \infty} \alpha_{n} = c.
\tag{14}
\end{equation}
\end{lemma}

\noindent 
{\bf Proof.}~ If $\limsup_{n \rightarrow \infty} \alpha_{n} = \infty$, then
exists a subsequence $\{m_{n}; n \in \mathbb{N} \}$ of $\mathbb{N}$ such that
$\lim_{n \rightarrow \infty} \alpha_{m_{n}} = \infty$. Thus, from (13), 
\[
\left(Y_{m_{n}} \right)_{n \in \mathbb{N}} 
= \left(\left(\frac{1}{\alpha_{m_{n}}}\right) \alpha_{m_{n}} Y_{m_{n}} \right)_{n \in \mathbb{N}} 
~\mbox{converges weakly to}~0 
\]
which contradicts (12). Thus, we conclude that $\limsup_{n \rightarrow \infty} \alpha_{n} < \infty$. 
Similarly, from (12) and (13), we conclude that $\liminf_{n \rightarrow \infty} \alpha_{n} > 0$. Therefore, we have
\begin{equation}
0 < \liminf_{n \rightarrow \infty} \alpha_{n} \leq \limsup_{n \rightarrow \infty} \alpha_{n} < \infty.  \tag{15}
\end{equation}

Let $\liminf_{n \rightarrow \infty} \alpha_{n} = c_{1}$ and $\limsup_{n \rightarrow \infty} \alpha_{n} = c_{2}$. 
Then, from (15), we see that $0 < c_{1} \leq c_{2} < \infty$. If $c_{1} < c_{2}$, 
then there exist two subsequences $\{l_{n}; n \in \mathbb{N} \}$ and 
$\{m_{n}; n \in \mathbb{N} \}$ of 
$\mathbb{N}$ such that $\lim_{n \rightarrow \infty} \alpha_{l_{n}} = c_{1}$ and 
$\lim_{n \rightarrow \infty} \alpha_{m_{n}} = c_{2}$ and hence, from (12), 
\[
\left(\alpha_{l_{n}} Y_{l_{n}} \right)_{n \in \mathbb{N}} 
~\mbox{converges weakly to}~c_{1}W_{1} ~~\mbox{and}~~
\left(\alpha_{m_{n}} Y_{m_{n}} \right)_{n \in \mathbb{N}} 
~\mbox{converges weakly to}~c_{2}W_{1}
\]
which contradicts (13). Therefore $c_{1} = c_{2} = c \in (0, \infty)$ and 
(14) holds. ~$\Box$

\vskip 0.2cm

\noindent {\bf Proof of Theorem 2.}~ Since $X \in \mbox{CLTSS}$, there exists a subsequence 
$\{m_{n}; ~n \in \mathbb{N} \}$ of $\mathbb{N}$ and a $\mathbf{B}$-valued 
random variable $U$ such that (5) holds. Since $X \notin \mbox{CLT}$, $U$ 
is nondegenerate and hence, $\mathbb{P}(U = 0) < 1$. 

If the conclusion of Theorem 2 is not true, then there exist a sequence 
$\{a_{n}; n \in \mathbb{N} \}$ of positive real numbers and a nondegenerate
$\mathbf{B}$-valued random variable $V$ (and hence, $\mathbb{P}(V = 0) < 1$) such that
\begin{equation}
 \left(S_{n}/a_{n} \right)_{n \in \mathbb{N}} ~ \mbox{converges weakly to}~ V.
 \tag{16}
\end{equation}

Since both $\mathbb{P}(U = 0) < 1$ and $\mathbb{P}(V = 0) < 1$, there exists a
$f \in \mathbf{B}^{*}$ such that $\mathbb{P}(f(U) = 0) < 1$ and $\mathbb{P}(f(V) = 0) < 1$. Write $W_{1} = f(U)$, $W_{2} = f(V)$, and
\[
Y_{n} = f\left(\frac{S_{n}}{\sqrt{n}} \right), ~\alpha_{n} = \frac{\sqrt{n}}{a_{n}}, ~ n \geq 1.
\]
Then, from (5) and (16) respectively, we have 
\begin{equation}
\left(Y_{m_{n}}\right)_{n \in \mathbb{N}} ~\mbox{converges weakly to}~W_{1} = f(U)
~\mbox{with}~\mathbb{P}\left(W_{1} = 0 \right) < 1
\tag{17}
\end{equation}
and
\begin{equation}
\left(\alpha_{n}Y_{n}\right)_{n \in \mathbb{N}} ~\mbox{converges weakly to}~W_{2} = f(V)
~\mbox{with}~\mathbb{P}\left(W_{2} = 0 \right) < 1.
\tag{18}
\end{equation}
Applying Proposition 1 of Song [14] and using the same argument used in the proof of Lemma 1, we have
\[
\mathbb{E}f(X) = 0 ~\mbox{and}~ \mathbb{E}f^{2}(X) < \infty
\]
and hence, by the classical CLT in real space, 
\begin{equation}
\left(Y_{n}\right)_{n \in \mathbb{N}} = \left(\frac{\sum_{i=1}^{n}f(X_{i})}{\sqrt{n}} \right)_{n \in \mathbb{N}}
~\mbox{converges weakly to}~N\left(0, \mathbb{E}f^{2}(X)\right).
\tag{19}
\end{equation}
Now (17) and (19) together ensure that
\begin{equation}
\left(Y_{n}\right)_{n \in \mathbb{N}} ~\mbox{converges weakly to}~W_{1} = f(U)
~\mbox{with}~\mathbb{P}\left(W_{1} = 0 \right) < 1.
\tag{20}
\end{equation}
By Lemma 2, it follows from (20) and (18) that
\[
\lim_{n \rightarrow \infty} \alpha_{n} = \lim_{n \rightarrow \infty} 
\frac{\sqrt{n}}{a_{n}} = c~~\mbox{for some}~ 0 < c < \infty.
\]
Thus, from (16), we get
\[
\left(\frac{S_{n}}{\sqrt{n}} \right)_{n \in \mathbb{N}}
= \left(\left(1/\alpha_{n}\right)\left(\frac{S_{n}}{a_{n}}\right) \right)_{n \in \mathbb{N}}~\mbox{converges weakly to}~ (1/c)V
\]
which contradicts $X \notin \mbox{CLT}$. This completes the proof of Theorem 2. ~$\Box$

\section{Conclusions}
In this note two theorems concerning the central limit theorem for subsequences of sums of 
i.i.d. $\mathbf{B}$-valued random variables are established. 
As direct implications of the two established theorems, the three questions raised in 
Section 1 are answered. 

Let $X$ be a $\mathbf{B}$-valued random variable. The interesting finding of this note is that 
$X \in \mbox{CLT}$ and $ X \in \mbox{CLTSS}$ are equivalent if $\mathbf{B}$ is of cotype $2$ 
space. The techniques used in the proof of this result relys upon our Lemma 1 and a theorem of 
Jain [8] (i.e., $\mathbf{B}$ is of cotype $2$ space if and only if $X \in \mbox{CLT}$ whenever 
$X$ is pre-Gaussian). We conjecture that if $\mathbf{B}$ is not of cotype $2$ space, then there 
exists a $\mathbf{B}$-valued random variable $X$ such that $\mathbb{E}X = 0$, 
$\mathbb{E}\|X\|^{2} < \infty$,  $X \notin \mbox{CLT}$, and $X \in \mbox{CLTSS}$.
We believe that the idea provided by Aldous [1] may help prove this conjecture. If this conjecture 
can be proven to be true, then the equivalence of $X \in \mbox{CLT}$ and $ X \in \mbox{CLTSS}$
can be used to characterize cotype $2$ space. 

\vskip 0.5cm

\noindent 
{\bf Author Contributions:} Conceptualization, D. Li and H.-Y. Liang; 
Methodology, D. Li and H.-Y. Liang; Writing-original draft, 
D. Li and H.-Y. Liang; Writing—review \& editing, D. Li and H.-Y. Liang. 
Two authors have read and agreed to the published version of the manuscript.

\vskip 0.3cm

\noindent 
{\bf Fundings:} The research of Deli Li was partially supported by a grant from the Natural Sciences 
and Engineering Research Council of Canada (RGPIN-2025-0573)
and the research of Han-Ying Liang
was partially supported by the National Natural Science Foundation of China (12071348) and  Natural 
Science Foundation of Shanghai, China (25ZR1401345).

\vskip 0.3cm

\noindent 
{\bf Data Availability Statement:} No data was used for the research described in this note.

\vskip 0.3cm

\noindent 
{\bf Conflicts of Interest:} The authors declare that they have no known competing financial interests that could have appeared to influence the work reported in this note.

\vskip 0.5cm

\noindent 
{\bf References}

\begin{enumerate}

\item Aldous, D. A characterization of Hilbert space using the central limit theorem. {\it J. Lond. Math. Soc.} {\bf 1976}, 14, 376-380.

\item Chow, Y. S.; Teicher, H.; Wei, C. Z.; Yu, K. F. Iterated logarithm for random subsequences. {\it Z. Wahrscheinlichkeitstheorie verw. Geb.} {\bf 1981}, 57, 235-251.

\item Fukuyama, K.; Takeuchi, Y. The law of the iterated logarithm for subsequences: a simple proof. {\it Lobachevskii J. Math.} {\bf 2008}, 29, 130-132.

\item Gut, A. On complete convergence in the law of large numbers for subsequences. {\it Ann. Probab.} {\bf 1985}, 13, 1286-1291.

\item Gut, A. Law of the iterated logarithm for subsequences. {\it Probab. Math. Stat.} {\bf 1986}, 7: 27-58.

\item Hoffmann-J{\o}rgensen, J.; Pisier, G. The law of large numbers 
and the central limit theorem in Banach spaces. {\it Ann. Probab.} {\bf 1976}, 4, 587-599.

\item Jain, N.C. Central limit theorem in a Banach space. In {\it Probability in Banach Spaces: Proceedings of the First International Conference on Probability in Banach Spaces, 20–26 July 1975, Oberwolfach}; Beck, A., Ed.; Lecture Notes in Mathematics, Vol. 526; Springer: Berlin/Heidelberg, {\bf 1976}, 113-130.

\item Jain, N.C. Central limit theorem and related questions in Banach spaces. In {\it Proc. Symp. in Pure Mathematics, Vol. XXXI; Amer. Math. Soc.}: Providence, {\bf 1977}, 55-65. 

\item Kwapi\'{e}n, S. Isomorphic characterisations of inner product spaces by orthogonal 
series with vector valued coefficients. {\it Studia Math.} {\bf 1972}, 44, 583-595.

\item Ledoux, M.; Talagrand, M. {\it Probability in Banach Spaces: Isoperimetry and Processes}; Springer-Verlag: Berlin/Heidelberg, {\bf 1991}; ISBN 978-3-642-20211-7. 

\item Li, D.; Klesov, O.; Stoica, G. On the central limit theorem along subsequences 
of sums of i.i.d. random variables. {\it Stat. Pap.} {\bf 2014}, 55, 1035-1045.

\item Qualls, C. (1977). The law of the iterated logarithm on arbitrary sequences
for stationary Gaussian processes and Brownian motion. {\it Ann. Probab.} {\bf 5}: 724-739.

\item Schwabe; Gut, A. On the law of the iterated logarithm for rapidly increasing subsequences. 
{\it Math. Nachr.} {\bf 1996}, 178, 309-320.

\item Song, L. A counterexample in the central limit theorem. 
{\it Bull. Lond. Math. Soc.} {\bf 1999}, 31, 222-230.

\item Weber, M. The law of the iterated logarithm on subsequences-characterizations. {\it Nagoya Math. J.} {\bf 1990}, 118, 65-97.

\end{enumerate}

\end{document}